\documentclass[a4paper,12pt]{article}
\usepackage{amsmath,amssymb,amsthm}
\newcommand{\C}{{\mathbb C}}
\newcommand{\N}{{\mathbb N}}
\newcommand{\R}{{\mathbb R}}

\newcommand{\abs}[2][\empty]{\ifx#1\empty\left|#2\right|%
\else#1\vert #2 #1\vert\fi}
\newcommand{\Cnt}[1][]{{\cal C}^{#1}}
\newcommand{\conv}{\star}
\newcommand{\csub}{\subset\subset}

\newcommand{\eps}{\varepsilon}

\newcommand{\norm}[2][\empty]{\ifx#1\empty\left\Vert#2\right\Vert%
\else#1\Vert #2 #1\Vert\fi}
\newcommand{\pseudonorm}[2][\empty]{\ifx#2.{\mathcal P}\else\ifx#1\empty{\mathcal P}(#2)%
\else{\mathcal P}#1( #2 #1)\fi\fi}
\newcommand{\restr}[2]{{#1}_{|#2}}
\newcommand{\supp}{\mathop{\mathrm{supp}}}
\newcommand{\test}{\mathcal D}

\newcommand{\val}{\mathop\mathrm v}

\newcommand{\Gen}{{\mathcal G}}
\newcommand{\GenC}{\widetilde\C}

\newcommand{\EMod}{{{\mathcal E}_M}}
\newcommand{\Null}{{\mathcal N}}

\newcommand{\sharpnorm}[2][\empty]{\abs[#1]{#2}_{\mathrm e}}

\newcommand{\GenLin}[1]{\Gen_{\mathcal L_{#1}}}

\newtheorem{thm}{Theorem}[section]
\newtheorem{lemma}[thm]{Lemma}
\newtheorem{prop}[thm]{Proposition}
\newtheorem{cor}[thm]{Corollary}

\theoremstyle{definition}
\newtheorem*{ack}{Acknowledgment}

\begin{document}
\title{Topological properties of regular generalized function algebras}
\author{H.~Vernaeve\footnote{Dept.\ Of Mathematics, Ghent University. E-mail: {\tt hvernaev@cage.ugent.be}}}
\date{}
\maketitle

\begin{abstract}
We investigate density of various subalgebras of regular generalized functions in the special Colombeau algebra $\Gen(\Omega)$ of generalized functions.
\end{abstract}

\section{Introduction}
M.~Oberguggenberger introduced the algebra $\Gen^\infty(\Omega)$ of regular generalized functions in order to develop a hypoelliptic regularity theory and hyperbolic propagation of singularities in the algebra $\Gen(\Omega)$ of Colombeau generalized functions \cite{O92}, where it takes over the role of the subalgebra of $\Cnt[\infty]$-regular functions in the space $\test'(\Omega)$ of distributions. It thus became the starting point of investigations of microlocal regularity in generalized function algebras (see \cite{DPS98,GH05,HK01,HOS06,NPS,S92} and the references therein). More recently, various other subalgebras of regular generalized functions have been considered, from the point of view of generalized analytic functions \cite{AFJ05}, kernel theorems \cite{D05}, propagation of singularities \cite{O06} and microlocal analysis \cite{D06}. We show that, in contrast with the situation of $\Cnt[\infty](\Omega)$ as a subalgebra of $\test'(\Omega)$ (and therefore maybe surprisingly), the subalgebra $\Gen^\infty(\Omega)$ and the subalgebras $\GenLin{a}(\Omega)$ considered in \cite{D05,D06} are not dense in the algebra $\Gen(\Omega)$. On the other hand, the subalgebra of sublinear or S-analytic generalized functions is dense in $\Gen(\Omega)$.

\section{Notations}
Let $\Omega\subseteq\R^d$ be open. By $K\csub\Omega$, we denote a compact subset of $\Omega$.\\
For $u\in\Cnt[\infty](\Omega)$, $K\csub\Omega$ and $\alpha\in\N^d$, let $p_{\alpha,K}(u):= \sup_{x\in K} \abs{\partial^\alpha u(x)}$. For $k\in\N$, let $p_{k,K}(u):= \max_{\abs\alpha =k} p_{\alpha,K}(u)$.\\
The special algebra of Colombeau generalized functions (see e.g.\ \cite{GKOS}) is $\Gen(\Omega) :=\EMod(\Omega)/\Null(\Omega)$, where
\begin{align*}
\EMod(\Omega) =\,&\big\{(u_\eps)_\eps \in\Cnt[\infty](\Omega)^{(0,1)}: (\forall
K\csub\Omega) (\forall\alpha\in\N^d) (\exists N\in\N)\\
&(p_{\alpha, K}(u_\eps)\le \eps^{-N}, \text{ for small }\eps)\big\}\\
\Null(\Omega) =\,&\big\{(u_\eps)_\eps \in\Cnt[\infty](\Omega)^{(0,1)}: (\forall
K\csub\Omega) (\forall\alpha\in\N^d) (\forall m\in\N)\\
&(p_{\alpha, K}(u_\eps)\le \eps^m, \text{ for small }\eps)\big\}.
\end{align*}
By $[(u_\eps)_\eps]$, we denote the generalized function with representative $(u_\eps)_\eps\in\EMod(\Omega)$.\\
The subalgebra $\Gen_c(\Omega)$ of compactly supported generalized functions consists of those $u\in\Gen(\Omega)$ such that for some $K\csub\Omega$, the restriction of $u$ to $\Omega\setminus K$ equals $0$ (as an element of $\Gen(\Omega\setminus K)$).\\
For $K\csub\Omega$, the algebra $\Gen^\infty(K)$ consists of those $u\in\Gen(\Omega)$ such that for one (and hence for each) representative $(u_\eps)_\eps$,
\[
(\exists N\in\N) (\forall\alpha\in\N^d) \big(p_{\alpha, K}(u_\eps)\le\eps^{-N}, \text{ for small }\eps\big).
\]
For $(z_\eps)_\eps\in\C^{(0,1)}$, the valuation $\val(z_\eps):= \sup\{b\in\R: \abs{z_\eps}\le \eps^b$, for small $\eps \}$ and the so-called sharp norm $\sharpnorm{z_\eps}:= e^{- \val(z_\eps)}$. For $u\in\Gen(\Omega)$, $P_{\alpha,K}(u) := \sharpnorm{p_{\alpha,K}(u_\eps)}$ ($\alpha\in\N^d$) and $P_{k,K}(u) := \sharpnorm{p_{k,K}(u_\eps)}$ ($k\in\N$), independent of the representative $(u_\eps)_\eps$ of $u$. The ultra-pseudo-seminorms $P_{\alpha,K}$ ($\alpha\in\N^d$, $K\csub\Omega$) determine a topology on $\Gen(\Omega)$ called sharp topology \cite{B90,G05,S92}. Then $u\in\Gen^\infty(K)$ iff $\sup_{k\in\N}P_{k,K}(u)<+\infty$. Further, the algebra $\Gen^\infty(\Omega):= \bigcap_{K\csub \Omega} \Gen^\infty(K)$ \cite{O92}.\\
For $K\csub\Omega$, the algebra $\GenLin{a}(K)$ of generalized functions of sublinear growth with slope smaller than $a>0$ ($a\in\R$) on $K$ consists of those $u\in\Gen(\Omega)$ such that for one (and hence for each) representative $(u_\eps)_\eps$,
\[(\exists a'<a) (\exists b\in\R) (\forall\alpha\in\N^d) (p_{\alpha,K}(u_\eps)\le\eps^{-a'\abs\alpha - b}, \text{ for small }\eps)\]
or, equivalently,
\[(\exists a'<a) (\exists c\in\R) (P_{\alpha,K}(u)\le c e^{a'\abs\alpha}, \forall \alpha\in\N^d),\]
which can still be expressed concisely by $\limsup_{k\to\infty} \frac{\ln P_{k,K}(u)}{k}< a$. Since $P_{\alpha,K}(uv)\le \max_{\beta\le\alpha} (P_{\beta,K}(u) P_{\alpha-\beta,K}(v))$ by Leibniz's rule, $\GenLin{a}(K)$ are subalgebras of $\Gen(\Omega)$.\\
For $a=0$, $\GenLin{0}(K) := \bigcap_{a>0}\GenLin{a}(K)$. Clearly, $\Gen^\infty(K)\subseteq \GenLin{0}(K)$.\\
Again, the algebras $\GenLin{a}(\Omega) := \bigcap_{K\csub\Omega} \GenLin{a}(K)$ ($a\ge 0$) \cite{D05,D06}. Clearly, $\Gen^\infty(\Omega)\subseteq \GenLin{0}(\Omega)$.\\
By definition, $u=[(u_\eps)_\eps]\in\Gen(\Omega)$ is sublinear \cite{AFJ05,PSV07} iff for each $K\csub\Omega$ and each $(x_\eps)_\eps\in K^{(0,1)}$, there exists $k\in\R$ and $(p_n)_{n\in\N}\in\R^\N$ such that $\lim_{n\to\infty} p_n + kn = \infty$ and for each $\alpha\in\N^d$, $\abs{\partial^\alpha u_\eps(x_\eps)}\le \eps^{p_{\abs\alpha}}$, for small $\eps$.
It can be shown \cite[Thm.~5.7]{AFJ05}, \cite[Thm.~10]{PSV07} that the algebra of sublinear generalized functions exactly contains those $u\in\Gen(\Omega)$ satisfying a natural condition of analyticity (called $S$-real analyticity in \cite{PSV07}). Sublinearity can still be characterized as follows by means of the algebras $\GenLin{a}(K)$:
\begin{lemma}
Let $u\in\Gen(\Omega)$. Then $u$ is sublinear iff for each $K\csub\Omega$, there exists $a>0$ ($a\in\R$) such that $u\in\GenLin{a}(K)$.
\end{lemma}
\begin{proof}
$\Rightarrow$: let $u$ be sublinear and suppose that there exists $K\csub\Omega$ such that $u\notin\GenLin{a}(K)$, for each $a>0$. Then we find $\alpha_n\in\N$ (for each $n\in\N$), $\eps_{n,m}\in(0,1/m)$ (for each $n,m\in\N$) (by enumerating the countable family $(\eps_{n,m})_{n,m}$, we can successively choose the $\eps_{n,m}$ such that they are all different) and $x_{\eps_{n,m}}\in K$ such that $\abs{\partial^{\alpha_n} u_{\eps_{n,m}}(x_{\eps_{n,m}})}>\eps_{n,m}^{-n\abs{\alpha_n}-n}$, for each $n,m\in\N$. Let $x_\eps\in K$ arbitrary if $\eps\in (0,1)\setminus \{\eps_{n,m}: n,m\in\N\}$. By assumption, there exist $k\in\R$, $(p_n)_{n\in\N}\in\R^\N$ and $N\in\N$ such that for each $\alpha\in\N^d$ with $\abs\alpha\ge N$, $\abs{\partial^\alpha u_\eps(x_\eps)}\le \eps^{p_{\abs\alpha}}\le \eps^{-k\abs\alpha}$, for small $\eps$. Since $u\in\Gen(\Omega)$, it follows that there exists $b\in\R$ such that for each $\alpha\in\N^d$, $\abs{\partial^\alpha u_\eps(x_\eps)}\le \eps^{-k\abs\alpha-b}$, for small $\eps$. This contradicts the fact that for $n\in\N$ with $n\ge k$ and $n\ge b$, $\lim_{m}\eps_{n,m} = 0$ and $\abs[]{\partial^{\alpha_n} u_{\eps_{n,m}}(x_{\eps_{n,m}})}>\eps_{n,m}^{-n\abs{\alpha_n}-n}$, $\forall m\in\N$.\\
$\Leftarrow$: let $K\csub\Omega$ and $(x_\eps)_\eps\in K^{(0,1)}$. By assumption, there exist $a,b\in\R$ such that for each $\alpha\in\N^d$, $p_{\alpha,K}(u_\eps)\le\eps^{-a\abs\alpha - b}$, for small $\eps$. Then, for $k:= a + 1$ and $p_n:= -an-b$, $\lim_{n} p_n +kn = \infty$ and for each $\alpha\in\N^d$, $\abs{\partial^\alpha u_\eps(x_\eps)}\le p_{\alpha,K}(u_\eps)\le\eps^{p_{\abs\alpha}}$, for small $\eps$.
\end{proof}

\section{$\Gen^\infty(\Omega)$ and $\GenLin{0}(\Omega)$}
Our method is based upon a quantitative version of an argument used in \cite[Thm.~1.2.3]{GKOS} (cf.\ also \cite[Prop.\ 1.6]{HOS06} and \cite{HV_pointwise}), which can in fact be traced back to \cite{Landau}.
\begin{prop}\label{prop_Landau}
Let $K\csub\Omega\subseteq\R^d$. Suppose that there exists $r\in\R^+$ such that for each $x\in K$, there exist $d$ line segments of length $r$ containing $x$ in linearly independent directions that are contained in $K$. Let $u\in\Gen(\Omega)$. If for some $k\in\N\setminus\{0\}$, $P_{k,K}(u) > P_{k-1,K}(u)$, then $P_{k,K}^2(u)\le P_{k-1,K}(u) P_{k+1,K}(u)$.
\end{prop}
\begin{proof}
Let first $k=1$. Let $x\in K$. Let $e_1,\dots,e_d\in\R^d$ be linearly independent unit vectors such that the line segments $[x, x+ \frac{r}{2} e_j]\subseteq K$. Denote the directional derivative in the direction $e_j$ by $\partial_{e_j}$. Let $a\in\R$, $a>0$. For $\eps\in (0,1)$, by Taylor's formula there exist $\theta_\eps\in [0,1]$ such that
\[
\partial_{e_j} u_\eps(x) = \eps^{-a} u_\eps(x + \eps^a e_j) - \eps^{-a} u_\eps(x) + \frac{\eps^a}{2} \partial_{e_j}^2 u_\eps(x + \eps^a \theta_\eps e_j).
\]
Hence for $\eps\le \eps_0$ (where $\eps_0$ does not depend on $x\in K$),
\[
\abs{\partial_{e_j} u_\eps(x)}\le
2\eps^{-a} \sup_{y\in K} \abs{u_\eps(y)} + \eps^a \sup_{y\in K} \abs[\big]{\partial_{e_j}^2 u_\eps(y)}\le 2\eps^{-a} p_{0,K}(u_\eps) + \eps^a p_{2,K}(u_\eps).
\]
Since $e_1$, \dots, $e_d$ are linearly independent, we can write $\partial_1$, \dots, $\partial_d$ as a linear combination (with coefficients independent of $\eps$ and $x$) of $\partial_{e_1}$, \dots, $\partial_{e_d}$. Thus there exists $C\in\R$ such that $p_{1,K}(u_\eps)\le C\eps^{-a}p_{0,K}(u_\eps) + C\eps^a p_{2,K}(u_\eps)$, and $P_{1,K}(u)\le \max(e^a P_{0,K}(u), e^{-a} P_{2,K}(u))$. Should $P_{2,K}(u)\le P_{0,K}(u)$, then letting $a\to 0$ would yield $P_{1,K}(u)\le P_{0,K}(u)$, contradicting the hypotheses. Hence $P_{2,K}(u) > P_{0,K}(u)$, and we can choose $a>0$ such that $e^{2a} = P_{2,K}(u)/P_{0,K}(u)$ (since the case $P_{0,K}(u)=0$ is trivial).\\
If $k\in\N\setminus\{0\}$ arbitrary, the same reasoning can be applied to all $\partial^\alpha u$ with $\abs\alpha = k-1$ instead of $u$.
\end{proof}

\begin{cor}(cf.\ \cite[Thm.~1.2.3]{GKOS})\label{cor_null_ideal}
Let $K\csub\Omega\subseteq\R^d$. Suppose that there exists $r\in\R^+$ such that for each $x\in K$, there exist $d$ line segments of length $r$ containing $x$ in linearly independent directions that are contained in $K$. Let $u\in\Gen(\Omega)$. If for some $k\in\N$, $P_{k,K}(u)=0$, then $P_{l,K}(u)=0$, $\forall l\ge k$.
\end{cor}
\begin{proof}
If $P_{k+1,K}(u)\ne 0$, then $P_{k+1,K}(u)^2\le P_{k,K}(u)P_{k+2,K}(u)=0$ by proposition \ref{prop_Landau}, a contradiction. The result follows inductively.
\end{proof}

\begin{prop}\label{prop_Geninfty_char}
Let $K\csub\Omega$ satisfy the hypothesis of proposition \ref{prop_Landau}. Let $u\in\GenLin{0}(K)$. Then $P_{k,K}(u)$ are decreasing in $k$, and
\[\Gen^\infty(K) = \GenLin{0}(K) = \{u\in\Gen(\Omega): P_{k,K}(u)\le P_{0,K}(u), \forall k\in\N\}.\]
In particular, $\Gen^\infty(K)$ is closed in $\Gen(\Omega)$.
\end{prop}
\begin{proof}
Let $u\in\Gen(\Omega)$. If $P_{k,K}(u)$ are not decreasing in $k$, then there exists $k\in\N\setminus \{0\}$ such that $P_{k,K}(u)> P_{k-1,K}(u) > 0$ by corollary \ref{cor_null_ideal}. Let $r:= P_{k,K}(u)/ P_{k-1,K}(u)>1$. By proposition \ref{prop_Landau}, $P_{k+1,K}(u)\ge r P_{k,K}(u)$ (in particular, $P_{k+1,K}(u) > P_{k,K}(u)$). Inductively, $P_{k+n,K}(u)\ge r^n P_{k,K}(u)$, for each $n\in\N$. Thus $\limsup_{n\to\infty} \frac{\ln P_{n+k,K}(u)}{n+k} \ge \limsup_{n\to\infty} \frac{\ln (r^n P_{k,K}(u))}{n+k} = \ln r > 0$, and $u\notin\GenLin{0}(K)$. In particular, $u\notin\Gen^\infty(K)$.\\
The fact that $\Gen^\infty(K)$ is closed follows by continuity of $P_{k,K}$.
\end{proof}

\begin{thm}
$\Gen^\infty(\Omega) = \GenLin{0}(\Omega)$ is closed in $\Gen(\Omega)$. In particular, $\Gen^\infty(\Omega)$ is not dense in $\Gen(\Omega)$.
\end{thm}
\begin{proof}
$\Gen^\infty(\Omega) = \bigcap_K \Gen^\infty(K)$, where $K$ runs over all compact subsets of $\Omega$ that are a finite union of $d$-dimensional cubes parallel with the coordinate axes (hence satisfying the hypothesis of proposition \ref{prop_Landau}), and similarly for $\GenLin{0}(\Omega)$. The conclusions follow from proposition \ref{prop_Geninfty_char}.
\end{proof}

\section{$\GenLin{a}(\Omega)$, $a>0$}
\begin{prop}\label{prop_GenLin_char}
Let $K\csub\Omega$ satisfy the hypothesis of proposition \ref{prop_Landau}. Let $a\in\R$, $a\ge 1$. Then
\begin{multline*}
\{u\in\Gen(\Omega): (\exists c\in\R) (P_{k,K}(u)\le c a^k, \forall k\in\N)\}\\
= \{u\in\Gen(\Omega): P_{k+1,K}(u)\le a P_{k,K}(u), \forall k\in\N\}.
\end{multline*}
In particular, this describes a closed subset of $\Gen(\Omega)$.
\end{prop}
\begin{proof}
Let $u\in\Gen(\Omega)$. If $P_{k+1, K}(u)> a P_{k,K}(u)$, for some $k\in\N$, then $P_{k,K}(u)>0$ by corollary \ref{cor_null_ideal}. Let $r:= P_{k+1, K}(u)/P_{k,K}(u) > a$. By proposition \ref{prop_Landau}, $P_{k+n,K}(u)\ge r^n P_{k,K}(u)$, for each $n\in\N$. Thus $\limsup_{n\in\N} P_{n,K}(u)/a^n \ge \limsup_{n\in\N} \frac{r^{n-k} P_{k,K}(u)}{a^n} = +\infty$.\\
The other inclusion is clear.
\end{proof}

\begin{thm}
Let $a\in\R$, $a>0$. Then $\GenLin{a}(\Omega)$ is not dense in $\Gen(\Omega)$.
\end{thm}
\begin{proof}
$\GenLin{a}(\Omega)\subseteq \bigcap_{K}\{u\in\Gen(\Omega): (\exists c\in\R) (P_{k,K}(u)\le c e^{ak},\forall k\in\N)\}=:\mathcal A$, where $K$ runs over all compact subsets of $\Omega$ that are a finite union of $d$-dimensional cubes parallel with the coordinate axes. The set $\mathcal A$ is closed by proposition \ref{prop_GenLin_char} and is a strict subset of $\Gen(\Omega)$.
\end{proof}

\section{Sublinear generalized functions}
In order to investigate the density of the algebra of sublinear generalized functions, we start with the following proposition (see also \cite[Prop.\ 4.3.1]{S92}):
\begin{prop}\label{prop_convoluted_repres}
Let $\psi\in\Cnt[\infty](\R^d)$ with $\psi(x)=0$ if $\abs x\ge 1$ and $\int_{\R^d}\psi =1$. Denote by $\psi_\eps(x):=\eps^{-d}\psi(x/\eps)$, for each $\eps\in (0,1)$. 
If $u=[(u_\eps)_\eps]\in\Gen(\Omega)$, then $\lim_{n\to\infty}[(u_\eps\conv\psi_{\eps^{n}})_\eps]=u$.
\end{prop}
\begin{proof}
Let $n\in\N$ and $K\csub\Omega$. Then $u_\eps\conv\psi_{\eps^{n}}(x)=\int_{\abs{t}\le \eps^n} u_\eps(x-t)\psi_{\eps^{n}}(t)\,dt$ is well-defined as soon as $d(x,\R^d\setminus\Omega)>\eps^{n}$. For small $\eps$, $d(K,\R\setminus\Omega)>\eps^n$ and thus $\restr{(u_\eps\conv\psi_{\eps^n})}{K}$ can be extended to a $\Cnt[\infty](\Omega)$-function. Independent of the extension, by the mean value theorem,
\begin{multline*}
p_{k,K}(u_\eps\conv\psi_{\eps^n} - u_\eps)
= \sup_{x\in K, \abs\alpha = k}\abs{(\partial^\alpha u_\eps) \conv \psi_{\eps^{n}}(x) - \partial^\alpha u_\eps(x)}\\
= \sup_{x\in K, \abs\alpha = k}\abs{\int_{\abs t\le \eps^n} (\partial^\alpha u_\eps(x-t) - \partial^\alpha u_\eps(x)) \psi_{\eps^{n}}(t)\,dt}
\le \eps^{n} p_{k+1,K+r}(u_\eps) \int_{\R^d}\abs{\psi}
\end{multline*}
for small $\eps$, where $r>0$ ($r\in\R$) such that $K+r = \{x\in\R^d: d(x,K)\le r\}\csub\Omega$.
\end{proof}

\begin{prop}\label{prop_dense}
Let $\mathcal A$ be the set of all $u=[(u_\eps)_\eps]\in\Gen_c(\Omega)$ for which
\[(\exists N\in\N) (\forall K\csub\Omega) (\forall k\in\N) (p_{k,K}(u_\eps)\le \eps^{-Nk-N},\text{ for small }\eps).\]
Then $\mathcal A$ is dense in $\Gen(\Omega)$.
\end{prop}
\begin{proof}
Let $u\in\Gen_c(\Omega)$. Then there exists a representative $(u_\eps)_\eps$ of $u$ and $L\csub\Omega$ such that $\supp u_\eps \subseteq L$, for each $\eps$. For each $K\csub\Omega$ and $k\in\N$,
\begin{multline*}
p_{k,K}(u_\eps\conv\psi_{\eps^n})
=\sup_{x\in K,\abs\alpha=k} \abs{u_\eps\conv\partial^\alpha(\psi_{\eps^n})}\\
\le \eps^{-nk} \sup_{x\in L}\abs{u_\eps(x)} \max_{\abs\alpha = k}\int_{\R^d} 
\abs{\partial^\alpha \psi} \le \eps^{-nk-1} \sup_{x\in L}\abs{u_\eps(x)},
\end{multline*}
for small $\eps$. Thus $[(u_\eps\conv\psi_{\eps^n})_\eps]\in\mathcal A$. By proposition \ref{prop_convoluted_repres}, $\mathcal A$ is dense in $\Gen_c(\Omega)$. Further, $\Gen_c(\Omega)$ is dense in $\Gen(\Omega)$ (for $u\in\Gen(\Omega)$, $u=\lim_{n\to\infty}u\chi_n$, where $\chi_n\in\test(\Omega)$ with $\chi_n(x)=1$, $\forall x\in K_n$, where $(K_n)_{n\in\N}$ is a compact exhaustion of $\Omega$).\end{proof}

\begin{thm}
The subalgebra of sublinear generalized functions is dense in $\Gen(\Omega)$.
\end{thm}
\begin{proof}
With the notations of proposition \ref{prop_dense}, $\mathcal A\subseteq \{u\in\Gen(\Omega): u$ is sublinear$\}$.
\end{proof}

\begin{ack}
We are grateful to D.\ Scarpal{\'e}zos for very useful discussions.
\end{ack}

\end{document}